\documentclass[11pt,reqno]{amsart}
\usepackage{amsfonts}
\usepackage{amsmath,amscd}
\usepackage{fullpage}
\usepackage{amssymb}
\usepackage{centernot} 
\usepackage{enumerate} 
\usepackage{parskip}
\usepackage{pb-diagram} 
\usepackage{mathrsfs, tikz-cd}
\usepackage[OT2,T1]{fontenc}
\usepackage{seqsplit}
\usepackage{enumitem}
\usepackage{color}
\usepackage{array}
\usepackage[utf8]{inputenc} 
\usepackage{verbatim}
\usepackage{url}

\DeclareSymbolFont{cyrletters}{OT2}{wncyr}{m}{n}
\DeclareMathSymbol{\Sha}{\mathalpha}{cyrletters}{"58}

\definecolor{refkey}{rgb}{1,1,1}
\definecolor{labelkey}{rgb}{1,1,1}
\definecolor{cite}{rgb}{0.9451,0.2706,0.4941}
\definecolor{ruri}{rgb}{0.0078,0.4022,0.8010}

\usepackage[%
bookmarks=true,bookmarksnumbered=true,%
colorlinks=true,linkcolor=ruri,citecolor=red%
 ]{hyperref}

\makeindex 

\def\F{{\rm \mathbb{F}}}
\def\Z{{\rm \mathbb{Z}}}

\def\Q{{\rm \mathbb{Q}}}

\def\v{{\rm {\bf v}}}
\def\k{{\rm {\bf k}}}
\def\G{{\rm \mathbb{G}}}

\def\C{{\rm \mathbb{C}}}

\def\T{{\rm \mathbb{T}}}
\def\P{{\rm \mathbb{P}}}

\def\m{{\rm \mathfrak{m}}}

\def\O{{\rm \mathcal{O}}}

\def\A{{\rm \mathbb{A}}}

\def\divv{{\rm div}}

\def\Aut{{\rm Aut}}

\def\Disc{{\rm Disc}}

\def\GL{{\rm GL}}
\def\Gal{{\rm Gal}}

\def\rk{{\rm rk}}

\def\Div{{\rm Div}}

\def\Hom{{\rm Hom}}
\def\End{{\rm End}}

\def\Frob{{\rm Frob}}

\numberwithin{equation}{section}

\newtheorem{theorem}{Theorem}
\newtheorem{lemma}[theorem]{Lemma}

\newtheorem{remark}[theorem]{Remark}
\newtheorem{definition}[theorem]{Definition}
\newtheorem{example}[theorem]{Example}
\newtheorem{conjecture}[theorem]{Conjecture}
\newtheorem{corollary}[theorem]{Corollary}
\newtheorem{proposition}[theorem]{Proposition}

\usepackage{marginnote,color}

\usepackage{tikz}

\makeatletter
\let\@@pmod\pmod
\DeclareRobustCommand{\pmod}{\@ifstar\@pmods\@@pmod}
\def\@pmods#1{\mkern4mu({\operator@font mod}\mkern 6mu#1)}
\makeatother

\begin{document}
\setlength{\parskip}{2pt}
\setlength{\parindent}{15.5pt}

\title{Sandpile groups of supersingular isogeny graphs}

\author{Nathana\"{e}l Munier}
\author{Ari Shnidman}
\address{Ari Shnidman, Einstein Institute of Mathematics, The Hebrew University of Jerusalem, Edmund J.\ Safra Campus, Jerusalem 9190401, Israel.\vspace*{-3pt}}
\email{ariel.shnidman@mail.huji.ac.il}
\begin{abstract}
Let $p$ and $q$ be distinct primes, and let $X_{p,q}$ be the $(q+1)$-regular graph whose nodes are supersingular elliptic curves over $\overline{\F}_p$ and whose edges are $q$-isogenies.  For fixed $p$, we compute the distribution of the $\ell$-Sylow subgroup of the sandpile group (i.e.\ Jacobian) of $X_{p,q}$ as $q \to \infty$.  We find that the distribution disagrees with the Cohen-Lenstra heuristic in this context. Our proof is via Galois representations attached to modular curves.  As a corollary of our result, we give an upper bound on the probability that the Jacobian is cyclic, which we conjecture to be sharp.   
\end{abstract}

\maketitle
\newcommand{\ShortTitle}{\A}
\vspace{-.25in}

\section{Introduction}
Attached to any finite undirected graph $X$ is a finite abelian group $J(X)$, called the Jacobian or sandpile group of $X$.\footnote{Other names for it are the critical group and the divisor class group. See Section \ref{sec:Jacobians} for the definition.}  One may think of $J(X)$ as the class group of the graph, in analogy with the divisor class group of an algebraic curve or the ideal class group of a number field.  

For any prime number $\ell$, let $J(X)[\ell^\infty]$ denote the subgroup of elements killed by some power of $\ell$.  As with the Cohen-Lenstra heuristics for ideal class groups \cite{CL}, we can ask about the distribution of the finite abelian $\ell$-group $J(X)[\ell^\infty]$, as $X$ varies over a countable family of graphs.\footnote{For the families we consider, the size of $J(X)$ grows, so we cannot make sense of the analogous question for the entire group $J(X)$.}  Recently, Wood \cite{WoodSandpile} has verified a Cohen-Lenstra-type heuristic for Erd\H{o}s-R\'enyi random graphs, and M\'esz\'aros proved analogous results for random regular graphs \cite{meszaros}.  

In this paper, we compute these distributions for certain families of Ramanujan graphs, namely supersingular isogeny graphs. 
To define these graphs, let $p$ and $q$ be distinct prime numbers, and  assume $p \equiv 1 \pmod{12}$, for simplicity.  Let $X_p$ be the set of isomorphism classes of supersingular elliptic curves $E$ over $\overline{\F}_p$, a set of size $n:=(p-1)/12$. See \cite[\S 5.4]{AEC} for background on supersingular elliptic curves. Each $E \in X_p$ contains $q+1$ distinct subgroups of order $q$, and hence admits $q+1$ degree $q$ isogenies $\phi \colon E \to E'$, up to isomorphism. Moreover, each $q$-isogeny $E \to E'$ admits a dual  $q$-isogeny $E'\to E$ in the opposite direction.  Let $X_{p,q}$ denote the undirected $(q+1)$-regular graph whose vertex set is $X_p$ and whose edges correspond to $q$-isogenies.  Set $J_{p,q}:=J(X_{p,q})$.    

Our main results determine the distribution of the finite abelian groups $J_{p,q}[\ell^\infty]$, when $p$ and $\ell$ are fixed and $q \to \infty$. Thus, the number of vertices is fixed, while the degree of the graph goes to infinity. Our proof makes use of the link between the graphs $X_{p,q}$ and the Galois representation attached to the modular curve $X_0(p)$, as we shall explain.  It would  be interesting to consider families where $q$ is fixed and $p \to \infty$, but this would require different methods. 

\subsection{Results}
When $p$ is fixed, the groups $J_{p,q}$ carry an extra structure which heavily influences their distribution.  Namely, they are each modules for a certain ring $\T$, called the Hecke algebra. To define $\T$, let $L_q$ be the Laplacian of $X_{p,q}$, viewed as a linear operator on the space of functions $\Hom(X_p, \Z)$. Note that as $q$ varies, the operators $L_q$ act on the same space of functions. Crucially, they commute with each other.  Define the submodule of degree-zero functions
\[M = \{f \in \Hom(X_p,\Z) \colon \sum_{E \in X_p} f(E) = 0\} ,\]
on which $L_q$ acts as well.  Then $\T$ is the commutative subring of $\End_\Z(M)$ generated by the endomorphisms $L_q$. It is known to be semisimple and of rank $n-1 = \rk_\Z M$ as a $\Z$-module, which means that $M$ is rank 1 (not necessarily free) as a $\T$-module.    The Jacobian $J_{p,q}$ is then a module for the Hecke algebra $\T$. In fact, $J_{p,q}$ is essentially the cokernel $M/L_qM$. More precisely,  for $\ell \nmid n$, we have 
\begin{equation}\label{eq:sandpile}
J_{p,q}[\ell^\infty] \simeq M_\ell/L_qM_\ell,
\end{equation}
where  $M_\ell = M\otimes_\Z \Z_\ell$ (see Proposition \ref{prop:groupstructure}).  Given this, one might naively guess that the groups $J_{p,q}[\ell^\infty]$ should be distributed as a random $\T$-module of the form $\T_\ell/a\T_\ell$, where $\T_\ell = \T \otimes_\Z \Z_\ell$ and $a \in \T_\ell$ is sampled according to Haar measure. However, this turns out to be incorrect, because it ignores the particular form of the elements $L_q$.

By definition $L_q = q+1 - T_q$, where $T_q \in \T$ is the image of the adjacency matrix of $X_{p,q}$, thought of as an element of $\End_\Z\, \Hom(X_p,\Z)$. Thus $L_q$ is (formally) the determinant of $x - 1$ acting on the free rank two $\T$-module $\T[x]/(x^2 - T_qx + q)$.  This rank two module arises naturally in the context of Hashimoto's edge-adjacency operator, whose characteristic polynomial computes the zeta function of the graph. This suggests a different model for the groups $J_{p,q}[\ell^\infty]$, namely as random cokernels of the form $\T_\ell/\det(g - 1)\T_\ell$, where $g$ is sampled from the group 
\[\G_\ell := \{g \in \GL_2(\T_\ell) \colon \det(g) \in \Z_\ell^\times\},\] 
endowed with its $\ell$-adic topology and its probability Haar measure.
 
Our first result states that for all but finitely many primes $\ell$, this is indeed the correct distribution. For an arbitrary $\T$-module $G$, define
\[\P(J_{p,q}[\ell^\infty] \simeq G) := \lim_{X \to \infty}\dfrac{\#\{q < X \colon J_{p,q}[\ell^\infty] \simeq G\}}{\#\{q < X\}},\]
which is the probability that $J_{p,q}[\ell^\infty]$ is isomorphic to $G$ as $\T$-modules (or $\T_\ell$-modules).  Let $\Disc(\T) \in \Z$ be the discriminant of the ring $\T$.
\begin{theorem}\label{thm:mainabstract}
Fix a prime $\ell$ not dividing $6n\Disc(\T)$, and let $G$ be a finite $\T_\ell$-module.  Let $\mu$ be the Haar probability measure on the group $\G_\ell$.  Then 
\[\P(J_{p,q}[\ell^\infty] \simeq G)  = \mu(g \in \G_\ell \colon \T_\ell/\det(g - 1)\T_\ell \simeq G).\]
\end{theorem}
Before discussing the proof of Theorem \ref{thm:mainabstract}, we address the natural follow-up  questions:
what is the distribution in Theorem \ref{thm:mainabstract}, in concrete terms?  The first thing to observe is that the distribution behaves  differently depending on the $\ell$-adic valuation of $q-1$. Indeed, the cokernel $\T_\ell/\det(g-1)\T_\ell$ measures the $\ell$-adic distance from 1 to the two eigenvalues of $g$. If $\det g$ (which we imagine is $q$) is itself $\ell$-adically close to $1$, then both of the eigenvalues can be close to 1, whereas if $\det g$ is not close to 1, then at most one eigenvalue can be close to $1$.

With this phenomenon in mind, we only give an explicit formula for the distribution as $q$ varies through primes $q \not\equiv 1\pmod{\ell}$; the case of $q \equiv 1\pmod\ell$ is more complicated.  We may write $\T_\ell \simeq \bigoplus_{i = 1}^t \O_i$, where each $\O_i$ is a finite free ring extension of $\Z_\ell$ of degree $d_i$.  Since $\ell$ does not divide the discriminant of $\T$,  each $\O_i$ is a discrete valuation ring with maximal ideal $\ell\O_i$ and residue field $\O_i/\ell\O_i$ of size $\ell^{d_i}$.  The following result computes the probability $\P(J_{p,q}[\ell^\infty] \simeq G)$ in Theorem \ref{thm:mainabstract}, in terms of the integers $d_1, \ldots, d_t$. Since $M_\ell$ is a rank one $\T_\ell$-module, this probability is 0 unless $G$ is isomorphic to  $G_\k := \bigoplus_{i = 1}^t \O_i/\ell^{k_i}\O_i$, for some tuple $\k = (k_i) \in \Z^t_{\geq0}$. 

\begin{theorem}\label{thm:mainexplicit}
Assume $\ell$ does not divide $6n\Disc(\T)$. Let $\delta_{ij}$ be Kronecker's delta function.  
As $q \to \infty$ varying through primes $q \not \equiv1\pmod\ell$, we have
\[\P(J_{p,q}[\ell^\infty] \simeq G_\k) = \frac{1}{\#G_\k}\prod_{i = 1}^t\left(1 - \frac{1}{\ell^{d_i} - 1}\right)^{\delta_{0k_i}}.\]
\end{theorem}

Note that $G_\k \simeq \bigoplus_{i = 1}^t (\Z/\ell^{k_i}\Z)^{d_i}$, as abstract abelian groups. Thus, one can easily determine from Theorem \ref{thm:mainexplicit} the explicit distribution of the abstract abelian groups $J_{p,q}[\ell^\infty]$.  To compute $\T_\ell$ (and hence the integers $d_i$), one can use sage or Magma.  See Section 6 for some worked out examples.

It is interesting to compare the distribution in Theorem \ref{thm:mainexplicit} with a naive Cohen-Lenstra-type heuristic. As explained in \cite{CLP},  because the Jacobian of a graph $X$ is naturally endowed with a perfect symmetric bilinear form, the Cohen-Lenstra heuristic in the setting of sandpile groups predicts that for each finite abelian group $G$, 
\[\P(J(X) \simeq G) \propto \frac{\#\{\mbox{perfect symmetric bilinear }  G \times G \to \C^\times\}}{\#G\#\Aut(G)}.\]
This is what Wood proves in the context of Erdos-Renyi random graphs.  In our case, $J_{p,q}$ is a rank one $\T$-module, and the number of perfect $\T$-linear pairings of a rank one $\T$-module is equal to the size of its automorphism group. Thus the naive Cohen-Lenstra heuristic would predict
\[\P(J_{p,q} \simeq G_\k) \propto \frac{1}{\#G_\k},\]
 which also agrees with the naive guess that $J(X_{p,q})[\ell^\infty]$ is modeled by the groups $\T_\ell/a\T_\ell$ for $a \in \T_\ell$.  This is nearly what we find in Theorem \ref{thm:mainexplicit}, except there is an extra factor of $1 - \frac{1}{\ell^{d_i}-1}$ for each trivial component of $G_\k$ (i.e.\ for each $i$ such that $k_i = 0$).  We see that for Jacobians of supersingular isogeny graphs, there is a slight and unexpected bias towards $\T$-modules with non-trivial components. 

For any fixed $t \geq 0$, one can give explicit formulas for the distribution of the groups $J_{p,q}$ as $q \to \infty$ varies through primes such that $v_\ell(q-1) = k$, but these formulas seem quite complicated for $k > 0$.  This is somewhat analogous to the complications that arise in the Cohen-Lenstra-Martinet heuristics when the ground field contains $\ell$-th roots of unity \cite{Garton}.  

One statistical quantity which we found accessible without any restriction on $q$, is the probability that $J_{p,q}[\ell^\infty]$ is cyclic.
\begin{theorem}\label{thm:introcylicsylow}
Let $\ell$ be a prime not dividing $6n\Disc(\T)$.  Recall $\T_\ell \simeq \bigoplus_{i= 1}^t \O_i$, with each $\O_i$ unramified over $\Z_\ell$ of degree $d_i$. Let $t_1$ be the number of factors with $d_i = 1$. Then the probability that $J_{p,q}[\ell^\infty]$ is a cyclic abelian group is
\[\frac{\ell-2 + t_1}{\ell-1}\prod_{i = 1}^{t}\left(1 - \frac{1}{\ell^{d_i}-1}\right) + \left(\frac{\ell^2 + (t_1 - 1)\ell - 1}{\ell^3 - 2\ell^2 + 1}\right)\prod_{i = 1}^{t}\left(1 - \frac{\ell^{d_i}}{\ell^{2d_i} -1}\right).\]
As a function of $\ell$, this expression is $1 - O(\ell^{-2})$, with implied constant depending only on $p$.
\end{theorem}

We can prove similar results for the subgroup $J_{p,q}[L^\infty]$, where $L$ is any squarefree integer coprime to $6n\Disc(\T)$.  We find that the $\ell$-Sylow subgroups are independent of each other, as one might expect (see Theorem \ref{thm:cyclicprop}). Taking $L \to \infty$, we deduce the following upper bound on the proportion of primes $p$ for which $J_{p,q}$ is cyclic. 

\begin{theorem}\label{thm:introcyclic}
Let $S = \prod_{\ell \mid 6n\Disc(\T)}\ell$ and let $J_{p,q}[1/S]$ be the prime-to-$S$ part of $J_{p,q}$. Then as $q \to \infty$, the probability that $J_{p,q}[1/S]$ is cyclic is at most
\[\displaystyle \prod_{\ell \notin S} \left(\frac{\ell-2 + t_1(\ell)}{\ell-1}\prod_{i = 1}^{t(\ell)}\left(1 - \frac{1}{\ell^{d_i}-1}\right) + \left(\frac{\ell^2 + (t_1(\ell) - 1)\ell - 1}{\ell^3 - 2\ell^2 + 1}\right)\prod_{i = 1}^{t(\ell)}\left(1 - \frac{\ell^{d_i}}{\ell^{2d_i} -1}\right)\right) > 0,\]
where $t(\ell)$ is the number of factors in $\T_\ell$ and $t_1(\ell)$ is the number of degree one factors of $\T_\ell$. In particular, this is also an upper bound for the probability that $J_{p,q}$ is cyclic.
\end{theorem}

We conjecture that this upper bound is also the correct lower bound, which would imply that a positive proportion of the groups $J_{p,q}[1/S]$ are cyclic.

\subsection{Proofs}
The first step in the proof of Theorem \ref{thm:mainabstract} is to identify a rank two $\T_\ell$-module $V_\ell$, \emph{which does not depend on $q$}, and operators $g_q \in \End_{\T_\ell}(V_\ell)$ such that $J_{p,q}[\ell^\infty] \simeq \T_\ell/\det(g_q -1)\T_\ell$.  

To find $V_\ell$, we use the deep connection between supersingular isogeny graphs and Galois representations that are familiar in number theory. Indeed, the algebra $\T$ is also known to act by correspondences on the modular curve $X_0(p)$; see e.g. \cite{Emerton02}. Moreover, there is a $(\T \otimes \Q)$-linear isomorphism between $M \otimes \Q$ and the space $S_2(\Gamma_0(p), \Q)$ of weight two cusp forms of level $\Gamma_0(p)$, which may be viewed as the space of regular differentials on $X_0(p)$.  This is proved either by using the Jacquet-Langlands correspondence or via theta series as in \cite{Emerton02}.  Let $J_0(p)$ be the Jacobian of $X_0(p)$, an abelian variety of dimension $n - 1$.  Then the Tate module \[V_\ell = \lim_{\leftarrow} J_0(p)[\ell^k]\] is free of rank two over $\T_\ell$ and admits an action $\rho_\ell \colon G_\Q \to \Aut_{\T_\ell}(V_\ell)$ by the absolute Galois group $G_\Q = \Gal(\bar \Q/\Q)$. In conjunction with (\ref{eq:sandpile}) and the $\T$-module isomorphism $M \otimes \Q \simeq S_2(\Gamma_0(p),\Q)$, the Eichler-Shimura relation \cite{ShimuraEichler} implies that for $\ell \nmid pN$, we have
\begin{equation}\label{eq:jacobian}
    J_{p,q}[\ell^\infty] \simeq \T_\ell/\det(\rho_\ell(\mathrm{Frob}_q) - 1)\T_\ell,
\end{equation}
where $\mathrm{Frob}_q$ is a choice of Frobenius automorphism at $p$ (see Proposition \ref{prop:eichlershimura}). Thus $V_\ell$ is the sought after $\T_\ell$-module.  The elements $\rho_\ell(\mathrm{Frob}_q)$ lie in $\G_\ell$ because the determinant of $\rho_\ell$ is the cyclotomic character, and hence $\det(\rho_\ell(\mathrm{Frob}_q)) = q$.  As an aside, it was observed in \cite[\S 4.3]{Li}  that for $\ell \nmid n$, we have
\[\#J_{p,q}[\ell^\infty] = \#J_0(p)(\F_q)[\ell^\infty],\]
which is a weaker version of (\ref{eq:jacobian}).  In fact, one deduces from (\ref{eq:jacobian}) that the groups $J_{p,q}[\ell^\infty]$ and $J_0(p)(\F_p)[\ell^\infty]$ are {\it not} isomorphic in general, despite having the same cardinality.

To prove that the Frobenius cokernels $\mu$-equidistribute, we need two more  ingredients. The first is the fact that $\rho_\ell$ maps $G_\Q$ surjectively onto $\G_\ell$. General results for Galois representation attached to modular curves imply that this surjectivity holds for all but finitely many $\ell$.   To get precise control of the allowed values of $\ell$, we use Ribet's strong result in the case where the level $p$ is prime \cite{RibetPac}, which proves surjectivity for all $\ell \nmid n \Disc(\T)$.  The second ingredient in our proof is Cebotarev's density theorem. 

The proofs of Theorems 2-4 amount to some elaborate $\ell$-adic computations. There may be a slicker proof of Theorem \ref{thm:mainexplicit} using more group theory, but we prefer our elementary proof, which may be helpful for readers interested in computing the full distribution, including primes $q \equiv 1\pmod\ell$.

\subsection{Remarks and further directions}
For $\ell \mid n\Disc(\T)$, the distribution of the groups $J_{p,q}[\ell^\infty]$  may very well differ from those in Theorem  \ref{thm:mainabstract} and \ref{thm:mainexplicit}. For example, if there is a mod $\ell$ congruence between two newforms of level $p$, then $\ell$ will divide $\Disc(\T)$ and the distribution will not agree with Theorem \ref{thm:mainexplicit}. In that case, the distribution could conceivably agree with the abstract distribution in Theorem \ref{thm:mainabstract}, but Ribet's results no longer apply, so we are not sure.  

It is natural to wonder about the distribution of the groups $J_{p,q}[\ell^\infty]$ with $q$ fixed and $p \to \infty$.  The authors are not sure what the distribution should be in this case.  What is needed is an $\ell$-adic version of the results of Serre \cite{SerreHecke} and Conrey-Duke-Farmer \cite{ConreyDukeFarmer}, but   controlling the $\ell$-adic norms of the terms in the trace formula seems difficult. 

\subsection{Acknowledgments}
The authors thank Amitay Kamber, as well as Ori Parzanchevski for a helpful conversation about Hashimoto's edge-adjacency operator. The first author visited the Hebrew University of Jerusalem virtually as part of an internship via the Erasmus program and ENS Rennes. The second author was supported by the Israel Science Foundation (grant No. 2301/20).

\section{Jacobians of graphs}\label{sec:Jacobians}
Let $X$ be a finite undirected graph with vertex set $V(X)$ and edge set $E(X)$. Self-loops and multi-edges are allowed. The Jacobian of $X$ is defined in terms of the group of divisors on $G$ 
\[\Div(X) = \left\{\sum_{v \in V(X)} a_v v \colon a_v \in \Z\right\},\] 
which is the free abelian group on the set $V(X)$. We may of course identify $\Div(X)$ with the group $\Hom(V(X),\Z)$ from the introduction.  

The {\em degree} of a divisor $\sum a_v v$ is the integer $\sum a_v$. We write $\Div^0(X) \subset \Div(X)$ for the subgroup of degree 0 divisors. Each function $f \colon V(X) \to \Z$ determines a  {\em principal divisor} 
\[\divv(f) = \sum_v \sum_{e = vw \in E(X)} (f(v) - f(w)) v.\]  Two divisors $D, D' \in \Div(X)$ are {\it linearly equivalent} if their difference is principal, or, in other words, if $D - D' = \divv(f)$ for some function $f$. Principal divisors have degree 0, so equivalent divisors have the same degree. 
\begin{definition}{\em 
The \emph{Jacobian} $J(X)$ is the group of linear equivalence classes in $\Div^0(X)$, i.e.\ the quotient of $\Div^0(X)$ by the subgroup of principal divisors.  }
\end{definition}

\begin{remark}{\em
The group of all functions $f \colon V(X) \to \Z$ is generated by the indicator functions $\delta_v(w) = \delta_{vw}$, for each $v \in V(X)$. The linear equivalence relation is therefore generated by the relations $\sum_{e = vw} (v-w)$, one for each $v$.}  
\end{remark}

\section{Jacobians of supersingular isogeny graphs}
Recall that $p \equiv 1 \pmod{12}$ is a prime and $X_p$ is the set of supersingular elliptic curves over $\bar\F_p$. Let $q$ be any prime different from $p$, and let $X= X_{p,q}$ be the supersingular $q$-isogeny graph with vertex set $X_p$.  Then $G$ is a $(q+1)$-regular graph on $n:=(p-1)/12$ vertices.

The $\Z$-module $\Div(X)$ is free of rank $n$, with basis given by the $n$ supersingular elliptic curves $E \in X_p$. The Hecke operator $T_q \colon \Div(X) \to \Div(X)$ sends $E$ to $\sum_{E'\sim E} E'$, where the sum is over every edge $E' \sim E$ in $X_{p,q}$. 

\begin{example}\label{ex:p=37q=2}{\em 
When $p = 37$, there are three supersingular elliptic curves, with $j$-invariants  $8$ and the roots $\alpha$ and $\bar\alpha$ of $x^2 -6x -6 \in \F_{37}[x]$.  For $q = 2$, we compute in sage that
\begin{itemize}
\item $T_2E_8 = E_8 + E_\alpha + E_{\bar\alpha}$,
\item  $T_2E_\alpha = E_8 + 0E_\alpha + 2E_{\bar\alpha}$, and  
\item $T_2E_{\bar\alpha} = E_8 + 2E_\alpha + 0E_{\bar \alpha}$. 
\end{itemize}
The Hecke operator $T_2$ is degree $3$, since the graph $X_{37,2}$ is $3$-regular.   }
\end{example}
In general, $T_q$ sends divisors of degree $n$ to divisors of degree $(q+1)n$.  In particular, it preserves the subgroup $\Div^0(X)$. The latter sits in an exact sequence 
\begin{equation}\label{eq:exact}
    0 \to \Div^0(X) \longrightarrow \Div(X) \stackrel{\deg}{\longrightarrow} \Z \to 0,
\end{equation}
of abelian groups. For varying primes $q$, the operators $T_q \in \End_\Z\Div(X)$ commute with each other.  They therefore generate a commutative algebra $\tilde\T$, called the Hecke algebra, which is itself free of rank $n$ as a $\Z$-module \cite{Emerton02}. Then (\ref{eq:exact}) is an exact sequence of $\tilde\T$-modules, with $T_q$ acting on $\Z$ by multiplication by $q+1$. Let $\T$ be the subalgebra of $\End_\Z\, \Div^0(X)$ generated by the action of the $T_q$. Then $\T$ is a quotient of $\tilde \T$ and has rank $n-1$.

\begin{lemma}
There is an isomorphism of $\T$-modules $J_{p,q} \simeq \Div^0(X)/(q + 1 - T_q)\Div(X)$.
\end{lemma}
\begin{proof}
From the definitions, the subgroup of principal divisors is precisely $(q+1 - T_q)\Div(X)$.
\end{proof}

The degree map $\Div(X) \to \Z$ does not quite admit a $\tilde\T$-module section, but there is a map $\Z \to \Div(X)$ sending $1$ to $\Delta := \sum_{E \in X_p} E$.  The subgroup $\Div^0(X) \oplus \Z \Delta$ has index $n$ inside $\Div(X)$, with cokernel isomorphic to $\Z/n\Z$.  
 
 \begin{proposition}\label{prop:groupstructure}
 We have $\#J_{p,q} =  \frac1n \det(q + 1 - T_q | \Div^0(X))$. More precisely, there is an isomorphism of $\T$-modules 
 \[J_{p,q} \simeq \mathrm{coker}(q + 1 - T_q | \Div^0(X))/\langle(q + 1 - T_q)(E)\rangle,\] 
 for any choice of $E \in X_p$.
 \end{proposition}
 \begin{proof}
This follows from applying the snake lemma to the diagram
\begin{center}
		\begin{tikzcd}
0 \arrow[r] & \Div^0(X) \oplus \Z\Delta \arrow[r] \arrow[d] & \Div(X) \arrow[d] \arrow[r] & \Z/n\Z \arrow[r]\arrow[d] & 0 \\
0 \arrow[r] & \Div^0(X)	 \arrow[r]& \Div^0(X) \arrow[r] & 0 \arrow[r] &0
		\end{tikzcd}\end{center}
whose vertical maps are $q+1 - T_q$.
 \end{proof}

Since $\T$ and $\Div^0(X)$ are free of rank $n-1$ as $\Z$-modules, the group $\Div^0(X)$ is a torsion-free rank one $\T$-module, though not necessarily free. For any prime $\ell$, let $\T_\ell = \T \otimes_\Z \Z_\ell$ and $\Div^0(X)_\ell = \Div^0(X) \otimes_\Z \Z_\ell$.  The $\Q_\ell$-algebra $\T_\ell \otimes \Q = \T \otimes_\Z \Q_\ell$ is a finite product $\bigoplus_{i=1}^t K_i$ of finite field extensions of $\Q_\ell$.  Assume from now on that $\ell$ does not divide the discriminant of $\T$. Then each $K_i$ is unramified of degree $d_i  \geq 1$ over $\Q_\ell$, and 
\[\T_\ell \simeq \bigoplus_{i = 1}^t \O_i,\] where $\O_i$ is the integral closure of $\Z_\ell$ in $K_i$.  

\begin{proposition}\label{prop:decomp1}
If $\ell$ does not divide $n \Disc(\T)$, then
\[J_{p,q}[\ell^\infty] \simeq \bigoplus_{i = 1}^t \O_i/(q+1 - T_q)\O_i,\]
where we view $T_q$ in $\O_i$ via the projection $\T_\ell \to \O_i$.
\end{proposition}

\begin{proof}
Since $\T_\ell$ is a product of discrete valuation rings, any rank one torsion-free module is free. After identifying $\Div^0(X)_\ell$ with $\T_\ell$, the action of $T_q$ is by left-multiplication. Since $T_q$ acts $\T_\ell$-linearly, it sends $\O_i$ to $\O_i$ and the result follows from Proposition \ref{prop:groupstructure} (since we also assume $\ell\nmid n$).  
\end{proof}

As explained in the introduction, we can realize the group $J_{p,q}[\ell^\infty]$ as $\T_\ell/\det(g_q-1)\T_\ell$ for some $\T_\ell$-linear map $g_q$ on a rank two $\T_\ell$-module $V_{\ell,q}$. However, the module we constructed was somewhat contrived, and in particular, depended on $q$.  To prove Theorem \ref{thm:mainabstract} we will find a {\it single} rank two $\T_\ell$-model $V_{\ell}$ and an operator $g_q$ on it (for each prime $q$) with the same properties.

We must first recall the connection to modular forms.  It is known that the algebra $\T$ is isomorphic to the Hecke algebra acting on the space of weight two cusp forms $S_2(\Gamma_0(p),\Z)$ with integer coefficients \cite[Thm.\ 3.1]{Emerton02}.  It follows that there is a bijection between the maximal ideals $\lambda_i$ of $\T_\ell$ and pairs $(f,\lambda)$, where $f = \sum a_n(f) q^n$ runs through a set of representatives for the $\Gal(\bar \Q/\Q)$-orbits of newforms in $S_2(\Gamma_0(p), \bar \Q)$, and $\lambda$ runs through the prime ideals in $\O_f$ above $\ell$, where $\O_f$ is the ring of integers of the number field generated by the coefficients $a_n(f)$. Since we assume $\ell \nmid \Disc(\T)$, the ring $\O_i$ is isomorphic to the completion $\O_{f,\lambda}$ of $\O_f$ at $\lambda$. Thus 
\[\T_\ell \simeq \bigoplus_{(f,\lambda)} \O_{f,\lambda},\]
where the sum is understood to be over the pairs $(f,\lambda)$ as before.  

The following description of $J_{p,q}[\ell^\infty]$ in terms of modular forms gives a convenient way to compute this group using any software which computes newforms and their Fourier coefficients.

\begin{proposition}\label{prop:decomp2}
If $\ell$ does not divide $n \Disc(\T)$, then
\[J_{p,q}[\ell^\infty] \simeq \bigoplus_{(f,\lambda)} \O_{f,\lambda}/(q+1 - a_q(f))\O_{f,\lambda},\]
where we view $a_q(f)$ in $\O_{f,\lambda}$ via the projection $\T_\ell \to \O_{f,\lambda}$.
\end{proposition}

\begin{proof}
We use the fact that $\Div^0(X)_\ell$ is isomorphic as $\T_\ell$-module to $S_2(\Gamma_0,\Z_\ell)$ \cite{Emerton02}. In fact, they are both free of rank 1, under the assumptions on $\ell$. Multiplication by $T_q$ on $\O_{f,\lambda}$ is given by multiplication by $a_p(f_i)$, since $f$ is a newform, and so the same must be true for the action of $T_q$ on $\O_i$. Thus, the result follows from Proposition \ref{prop:decomp1}.
\end{proof}
Let $X_0(p)$ be the modular curve parameterizing elliptic curves with a cyclic subgroup of order $p$. This is a smooth projective algebraic curve over $\Q$.  Let $J_0(p)$ be its Jacobian, an abelian variety over $\Q$ of dimension equal to the genus of $X_0(p)$, which is $g = n-1$.  For each $k$, let $J_0(p)[\ell^k]$ denote the group of $\ell^k$-torsion points in $J_0(p)$, which is a finite $G_\Q$-module, isomorphic as a group to $(\Z/\ell^k\Z)^{2g}$.  These $G_\Q$-modules form an inverse system under the multiplication-by-$\ell$ maps.  Let $V_\ell = \lim_k J_0(p)[\ell^k]$ be the inverse limit.\footnote{In the literature, $V_\ell$ is more often written as $T_\ell$ or $T_\ell J_0(p)$, but we have already used up this notation.}
Then $V_\ell$ is free of rank $2g$ over $\Z_\ell$ and carries a $\Z_\ell$-linear action of $G_\Q$.  

The Tate module $V_\ell$ has even more structure. The Hecke algebra $\T$ acts by correspondences on the curve $X_0(p)$, and hence acts by endomorphisms on $J_0(p)$. It is known that $\T = \End(J_0(p))$, since $p$ is prime \cite{MazurEisenstein}, but we will only use the containment $\T \subset \End(J_0(p))$. It follows that $V_\ell$ is a rank two $\T_\ell$-module. If $\ell \nmid \Disc(\T)$ or if $\ell > 2$, then Mazur showed that $V_\ell$ is moreover a free $\T_\ell$-module \cite[II.15-17]{MazurEisenstein}.  Thus, the $G_\Q$-action can be thought of as a representation
\[\rho_\ell \colon G_\Q \longrightarrow \Aut_{\T_\ell}(V_\ell)\] or
\[\rho_\ell \colon G_\Q \longrightarrow \GL_2(\T_\ell),\] 
if either $\ell \nmid \Disc(\T)$ or $\ell > 2$. The Galois representation $V_\ell$ is unramified at all primes $p \nmid N\ell$, so that the action of an arithmetic Frobenius element $\mathrm{Frob}_q \in G_\Q$ is well-defined.

\begin{proposition}\label{prop:eichlershimura}
For $\ell \nmid n\cdot \Disc(\T)p$, there is an isomorphism of $\T_\ell$-modules 
\[J_{p,q}[\ell^\infty] \simeq \T_\ell/\det(\rho_\ell(\mathrm{Frob}_q) - 1)\T_\ell\]
where $\mathrm{Frob}_q \in G_\Q$ is a Frobenius element for $q$.     
\end{proposition}

\begin{proof}
If $\lambda$ is a maximal ideal in $\T$ and $k \geq 1$, let $J_0(p)[\lambda^k]$ denote the $G_\Q$-module of points in $J_0(p)(\bar \Q)$ killed by all elements in the ideal $\lambda^k$. Define $V_\lambda = \lim_k J_0(p)[\lambda^k]$, the $\lambda$-adic Tate module of $J_0(p)$.  Then we have a decomposition of $G_\Q$-representations:
\[V_\ell = \bigoplus_{(f,\lambda)} V_\lambda,\]
where $(f,\lambda)$ varies over orbits of newforms $f$ and prime ideals $\lambda$ in $\O_f$ lying above $\ell$, as before. It is known that $V_\lambda$ is free of rank two over $\O_i \simeq \O_{f,\lambda}$, and the determinant of $V_\lambda$ is the cyclotomic character $\chi$.  By the Eichler-Shimura relation \cite{ShimuraEichler}, the action of a Frobenius element $\mathrm{Frob}_q \in G_\Q$ on $V_\lambda$ has trace equal to $T_q \in \O_i$. Now, $T_q$ acts by $a_q(f)$ on $V_\lambda$ and $\chi(\mathrm{Frob}_q) = q$.  Thus, the characteristic polynomial of $\mathrm{Frob}_q$ on the rank two $\O_i$-module $V_\lambda$ is $x^2 - a_q(f)x + q$. We find that $\det(\rho_\ell(\mathrm{Frob}_q) - 1) = q+1 - a_q(f)$. The result now follows from Proposition \ref{prop:decomp2} and summing over all $(f,\lambda)$.  
\end{proof}

It is interesting to compare Proposition \ref{prop:eichlershimura} with the following result of Hashimoto \cite[4.3]{Li}.
 
\begin{theorem}\label{cor:hashimoto}
Let $X_0(p)$ be the modular curve of level $\Gamma_0(p)$ and let $J_0(p)$ be its Jacobian. Then \[n \cdot \#J_{p,q} = \#J_0(p)(\F_q)\]
\end{theorem}

Let us sketch an alternative proof of this theorem. If $A$ is an abelian variety over $\F_q$, then $\#A(\F_q) = \det(\mathrm{Fr} - 1)$ where $\mathrm{Fr}$ is the action of the geometric Frobenius on the $\ell$-adic Tate-module of $A$ \cite{Mumford}. Moreover, if $A$ is the reduction of an abelian variety $\tilde A$ over $\Q$, then this action agrees with the action of  $\Frob_q$ on the $\ell$-adic Tate module of $\tilde A$.    Thus, the theorem follows from Propositions \ref{prop:groupstructure} and \ref{prop:eichlershimura}, and some additional care to treat the primes $\ell \mid n\Disc(\T)$.  

Note that $J_0(p)(\F_q)$ always contains a point of order $n = (p-1)/12$, namely the reduction $D$ of the cuspidal divisor $(0) - (\infty)$ on the modular curve. Theorem \ref{cor:hashimoto} says that the finite abelian groups $J_{p,q}$ and $J_0(p)(\F_q)/\langle D \rangle$ have the same cardinality. However, these two groups are not isomorphic in general as the former is a rank 1 $\T$-module whereas the latter is not. In any case, note that $X_{p,q}$ describes $q$-isogenies of (supersingular) elliptic curves in characteristic $p$, whereas $X_0(p)(\F_q)$ describes $p$-isogenies of elliptic curves in characeristic $q$. Thus, Hashimoto's formula can be viewed as a kind of reciprocity law between their Jacobians.

\section{Proof of Theorem \ref{thm:mainabstract}}\label{sec:proofThm1}
Let $L_q  = q+1 - T_q\in \T_\ell$ be the Laplacian operator for the graph $X_{p,q}$. As in the proof of Proposition \ref{prop:eichlershimura}, we have 
\[L_q = \det(\rho_\ell(\mathrm{Frob}_q) - 1).\]  Recall that $\T_\ell \simeq \bigoplus_{i = 1}^t \O_i$, where each $\O_i$ is a discrete valuation ring.  Let 
\[v \colon (\T_\ell \otimes \Q)^\times \to \bigoplus_{i = 1}^t \Z\] 
be the homomorphism which is the discrete valuation on each factor. By Proposition \ref{prop:eichlershimura}, the group $J_{p,q}[\ell^\infty]$ depends only on the tuple $v(L_q)$. Indeed, 
\[J_{p,q}[\ell^\infty] \simeq \bigoplus_{i=1}^t \O_i/\ell^{k_i}\O_i,\]
where $v(L_q) = (k_i)_{i = 1}^t$.   

Now let $G$ be the finite  $\T_\ell$-module in the Theorem.   Recall that 
\[\P(J_{p,q}[\ell^\infty] \simeq G) = \lim_{X \to \infty}\dfrac{\#\{p < X \colon J_{p,q}[\ell^\infty] \simeq G\}}{\#\{p < X\}},\]
and $\P(J_{p,q}[\ell^\infty] \simeq G) = 0$ unless $G = \bigoplus_{i=1}^t \O_i/\ell^{k_i}\O_i$, for some $\k = (k_i)\in \Z^t_{\geq 0}$. So we assume $G = G_\k$, for some such $\k$. 

Consider the ideal $I_\k = \{a \in \T_\ell \colon v(a) \geq \k\}$ in $\T_\ell$. By Proposition \ref{prop:eichlershimura}, there is an isomorphism $J_{p,q}[\ell^\infty] \simeq G$ if and only if the image of $\rho_\ell(\Frob_q)$ in $\GL_2(\T_\ell/I_\k)$ has 1 as an eigenvalue and the image of $\rho_\ell(\Frob_q)$ in $\GL_2(\T_\ell/I_{\bf{w}})$ does not have 1 as an eigenvalue for any $\bf{w} > \k$. In particular, we can detect whether $J_{p,q}[\ell^\infty] \simeq G$ from the image of $\rho_\ell(\Frob_q)$ in $\GL_2(\T_\ell/I_{\k+\bf{1}})$, where $\k + {\bf{1}} = (k_i + 1)_{i = 1}^t$.  

A result of Ribet \cite{RibetPac} states that if $\ell\nmid 6n \Disc(\T)$, then the image of the Galois representation $\rho_\ell \colon G_\Q \to \GL_2(\T_\ell)$ is precisely the group   
\[\G_\ell := \{g \in \GL_2(\T_\ell) \colon \det(g) \in \Z_\ell^\times\}\]
from the introduction. Hence the image of $\rho_\ell(G_\Q)$ in $\GL_2(\T_\ell/I_{\k+{\bf{1}}})$ is  
\[\G_\ell({\k+{\bf{1}}}) :=\{g \in \GL_2(\T_\ell/I_{\k+{\bf{1}}}) \colon \det(g) \in \Z_\ell^\times({\k+{\bf{1}}})\},\]
where $\Z_\ell^\times({\k+{\bf{1}}})$ is the image of $\Z_\ell^\times \to \T_\ell^\times \to (\T_\ell/I_{\k+{\bf{1}}})^\times$. Let $L$ be the finite Galois extension of $\Q$ which is the fixed field of the kernel of $\rho_\ell \colon G_\Q  \to \GL_2(\T_\ell/I_{\k+{\bf{1}}})$. Thus $\Gal(L/\Q) \simeq \G_\ell({\k+{\bf{1}}})$. Applying the Cebotarev density theorem to $L/\Q$, we find that $\P(J_{p,q}[\ell^\infty] \simeq G)$ is equal to the  proportion of elements of the finite group $\G_\ell({\k+{\bf{1}}})$ whose reduction in $\GL_2(\T_\ell/I_\k)$ has 1 as an eigenvalue, but whose reduction in $\GL_2(\T_\ell/I_{\bf{w}})$ does not have $1$ as an eigenvalue, for all tuples $\k < \bf{w} < \k+\bf{1}$.  By definition of the Haar measure on $\G_\ell$, this is equal to the probability that a random $g \in \G_\ell$ satisfies $\T_\ell/\det(g-1)\T_\ell \simeq G$, as desired. 

\section{Computing the distribution}

In this section we prove Theorems \ref{thm:mainexplicit}, \ref{thm:introcylicsylow}, and \ref{thm:introcyclic}.  Fix a prime $\ell$ and integer $d \geq 1$, and let $\lambda = \ell^d$.  Let $\O$ be the (unique) degree $d$ unramified cyclic extension of $\Z_\ell$. Thus, $\O$ is a discrete valuation ring containing $\Z_\ell$, with maximal ideal $\m = \ell \O$ and with residue field $\O/\m$ isomorphic to the finite field $\F_\lambda$.  The maximal ideal of $\Z_\ell$ is $\ell \Z_\ell$, so that its residue field $\Z_\ell /\ell\Z_\ell \simeq \F_\ell$ is naturally a subfield of $\O/\m \simeq \F_\lambda$. For any $k \geq 1$, we have an inclusion of rings $\Z_\ell/\ell^k \hookrightarrow \O/\m^k$. For $D \in \O/\m^k$, define 
\[N(k,D) = \#\{g \in \GL_2(\O/\m^k) \colon \det(g-1) = 0 \mbox{ and } \det(g) = D\}.\]

\begin{proposition}\label{prop:1eig}
Fix $k \geq 1$ and $D \in (\Z_\ell/\ell^k\Z_\ell)^\times$.  If $v(1 - D) = 0$, then 
\[N(k,D) = \lambda^{2k-1}(\lambda + 1).\]
\end{proposition}

\begin{proof}
For this proof, we abuse the usual notation slightly and consider the ``reduced valuation''
\[v \colon \O/\m^k \to \{0, 1, \ldots, k\}\] 
defined by the formula $\pi^{-1}((a)) = \m^{v(a)}$, where $\pi \colon \O \to \O/\m^k$ is the reduction map and $(a)$ is the ideal generated by $a$.

We wish to compute the number 
\[N(k,D) = \#\left\{ \begin{pmatrix} a & b \\ c& d\end{pmatrix} \in \GL_2(\O/\m^k) \colon ad -bc = D \mbox{ and } 1 - a - d + D = 0\right\}.\] 
Now, if $a$ is fixed, then $d$ is determined by the formula  $d = 1 - a +D$. Note that $ad - D = (1-a)(a-D)$. 
Thus
\[
N(k,D) = \sum_{a \in \O/\m^k} \sum_{\substack{b \in \O/\m^k\\v(b) \leq v((1-a)(a-D))}}\sum_{\substack{c \in \O/\m^k\\bc = (1-a)(a-D)}}1
\]
\begin{lemma}
If $b, M \in \O/\m^k$ satisfy $v(b) \leq v(M)$ then the number of elements $c \in \O/\m^k$ such that $bc = M$ is $\lambda^{v(b)}$.
\end{lemma}
\begin{proof}
If $M = 0$, then $c$ satisfies $bc = 0$ if and only if $v(c) \geq k - v(b)$, and there are $\lambda^{k - (k - v(b))} = \lambda^{v(b)}$ such elements. If $M \neq 0$, then $c$ must have valuation $v(M) - v(b)$, and there are $\lambda^{k - v(M) + v(b) -1}(\lambda-1)$ such elements. Multiplying such elements by $b$, we are equally likely to obtain any element of valuation $v(M)$.  Thus, after dividing by the $\lambda^{k-v(M)- 1}(\lambda-1)$ elements of valuation $v(M)$, we find that $\lambda^{v(b)}$ elements $c$ satisfy $bc = M.$  
\end{proof}
By the Lemma, 
we have
\begin{equation}\label{eq:sum}
N(k,D) = \sum_{a \in \O/\m^k} \sum_{\substack{b \in \O/\m^k\\v(b) \leq v((1-a)(a-D))}}\lambda^{v(b)}.
\end{equation}
Since $v(1-D) = 0$, we can have $v(1-a) > 0$ and we can also have $v(a-D) > 0$ but never both. Moreover, $(1-a)(a-D) = 0$ if and only if $a = 1$ or $a = D$. Thus, we compute
\begin{align*}
N(k,D) = &\sum_{j = 1}^k\#\{a \colon v(1-a) = j\}\sum_{i = 0}^j\#\{b \colon v(b) = i\}\lambda^i + \\
&\sum_{j = 1}^k \#\{a \colon v(a-D) = j\} \sum_{i = 0}^j\#\{b \colon v(b) = i\}\lambda^i + \lambda^{2k-2}(\lambda-2)(\lambda-1)
\end{align*}
\begin{lemma}\label{lem:counts}
We have 
\[\#\{b \colon v(b) = i\}\lambda^i =\begin{cases}
\lambda^{k-1}(\lambda-1) & \mbox{ if } i < k\\
\lambda^k & \mbox{ if } i = k
\end{cases}. \]
\end{lemma}
\begin{proof}
Indeed, the number of $b \in \O/\m^k$ of valuation $i$ is $\lambda^{k-1-i}(\lambda-1)$ if $i < k$ and $1$ if $i = k$.  
\end{proof}
Separating the contributions from $a = 1$ and $a = D$, we have:
\[N(k,D) = 2(k+1)\lambda^k - 2k\lambda^{k-1} + \lambda^{2k-2}(\lambda-2)(\lambda-1) + A\]
where 
\[A = \sum_{j = 1}^{k-1}\#\{a \colon v(1-a) = j\}(j+1)\lambda^{k-1}(\lambda-1) + \sum_{j = 1}^{k-1} \#\{a \colon v(a-D) = j\} (j+1)\lambda^{k-1}(\lambda-1)\]
Using Lemma \ref{lem:counts}, we compute  
\[A = 2\lambda^{2k-1}(\lambda-1)^2\sum_{j = 2}^{k} \dfrac{j}{\lambda^j}\]
Combining the formula
\[\sum^n_{i = 0} i a^i = \dfrac{a - a^{n+1}}{(1-a)^2} - \dfrac{na^{n+1}}{1-a},\]
we eventually find that 
\[A = 2\lambda^{2k} - 2\lambda^k - 2(\lambda-1)k\lambda^{k-1} - 2(\lambda-1)^2\lambda^{2k-2}.\]
Thus, 
\begin{align*}
N(k,D) =& 2(k+1)\lambda^k - 2k\lambda^{k-1} + \lambda^{2k-2}(\lambda-2)(\lambda-1) 2\lambda^{2k} - 2\lambda^k - 2(\lambda-1)k\lambda^{k-1} - 2(\lambda-1)^2\lambda^{2k-2}\\
=& \lambda^{2k} + \lambda^{2k-1}\\
=& \lambda^{2k-1}(\lambda+1),
\end{align*}
as desired.
\end{proof}
For any $D \in \O/\m^{k+1}$, define the set
\[\mathcal{M}(k,D) = \{g \in \GL_2(\O/\m^{k+1}) \colon v(\det(g-1)) = k \mbox{ and } \det(g) = D\}\]
and let $M(k,D) = \#\mathcal{M}(k,D)$.
\begin{proposition}\label{prop:1eigpart2}
Fix $k \geq 0$ and $D \in (\O/\m^{k+1})^\times$. If $v(1-D) = 0$, then 
\[M(k,D) = 
\begin{cases} \lambda(\lambda-2)(\lambda+1) & \mbox{ if } k = 0\\
\lambda^{2k+1}(\lambda-1)(\lambda+1) & \mbox{ if } k > 0\\
\end{cases}
.
\]
\end{proposition}

\begin{proof}
First assume $k = 0$. Then $M(0,D)$ is the number of $g \in \GL_2(\F_\lambda)$ of determinant $D$ such that $\det(g-1) \neq 0$.  Now, the number of $g$ with determinant equal to $D$ is
\[\dfrac{\#\GL_2(\F_\lambda)}{\lambda-1} = (\lambda+1)\lambda(\lambda-1).\]
Thus, $M(0,D) = (\lambda+1)\lambda(\lambda-1) - N(1,D)$, and by Proposition \ref{prop:1eig}, this is 
\[(\lambda-1)\lambda(\lambda+1) - \lambda(\lambda+1) = \lambda(\lambda+1)(\lambda-2).\]
Now assume $k \geq 1$.  Let $\pi \colon \GL_2(\O/\m^{k+1}) \to \GL_2(\O/\m^k)$.  Recall that $N(k,\pi(D))$ is the cardinality of the set 
\[\mathcal{N}(k,\pi(D)) = \{h \in \GL_2(\O/\pi^k) \colon \det(h) = \pi(D) \mbox{ and }\det(h-1) = 0\}.\]
Then $\pi$ sends $\mathcal{M}(k,D)$ to $\mathcal{N}(k,\pi(D))$. Conversely, given $h \in \mathcal{N}(k,\pi(D))$, there is at least one $g \in \pi^{-1}(h)$ with $\det(g) = D$. Any other element of $\pi^{-1}(h)$ has the form $g + \ell^k M$ for some matrix $M$.  The determinant of $g+ \ell^k M$ is equal to
\[(a+\ell^k m_1)(d + \ell^k m_4) - (b + \ell^km_2)(c + \ell^k m_3) = \det(g) + \ell^k(am_4 + dm_1 - cm_2 - bm_3) \]    
and hence equals $D = \det(g)$ if and only if $M$ lies in the kernel of a certain non-zero linear functional $\mathrm{Mat}_2(\F_\lambda) \to \F_\lambda$. There are therefore $\lambda^3$ pre-images $g' \in \pi^{-1}(h)$ with $\det(g') = D$. Not all of them satisfy $v(\det(g-1)) = k$, but we have
\[M(k,D) = \lambda^3N(k,\pi(D)) - N(k+1,D).\]
By Proposition \ref{prop:1eig}, we have
\begin{align*}
M(k,D) =& \lambda^{2k+2}(\lambda+1) - \lambda^{2(k+1)-1}(\lambda+1)\\
=& \lambda^{2k+1}(\lambda(\lambda+1) - \lambda-1)\\
=& \lambda^{2k+1}(\lambda-1)(\lambda+1),
\end{align*}
as claimed.
\end{proof}

Now let $\k = (k_i)  \in \Z^t_{\geq0}$. Recall $\T_\ell = \bigoplus_{i = 1}^t \O_i$ and each $\O_i$ is a finite unramified extension of $\Z_\ell$ of degree $d_i \geq 1$. The residue field $\O_i/\m_i$ of $\O_i$ is isomorphic to $\F_{\lambda_i}$, where $\lambda_i = \ell^{d_i}$.  Recall that $\k$ determines a finite $\ell$-primary $\T_\ell$-module  $G_\k = \bigoplus_{i = 1}^t \O_i/\ell^{k_i}$. If $g \in \G_\ell$, then $\T_\ell/\det(g-1)\T_\ell \simeq G_\k$ if and only if $v(\det(g-1)) = \k$, where here $v \colon \T_\ell \to \Z^t$ is the combined valuation from Section \ref{sec:proofThm1}. 

Recall the ideals $I_{\k+1}$ and the sets $\Z_\ell^\times(\k+1)$. 
Having fixed $\k$ for the time being, we will abuse notation and write $v \colon \T_\ell/I_{\k + 1} \to \prod_{i = 1}^t \{0,1,\ldots, k_i+1\}$ for the ``combined reduced valuation''.   Let $k = \max_i k_i$, and observe that we may view $\Z_\ell/\ell^{k+1} \Z_\ell$ as a subring of $\T_\ell/I_{\k+1}$.  Define $\Z_\ell^{\diamondsuit}(\k+1)$ to be the subset of elements $D$ of this subring such that $v(D(1-D)) = 0$.
Now define 
\[\mathcal{M}(\k) = \left\{g \in \GL_2(\T_\ell/I_{\k + 1}) \colon \det(g) \in \Z_\ell^{\diamondsuit}(\k+1) \mbox{ and } v(\det(g-1)) = \k\right\}.\]
Finally, set $M(\k) = \#\mathcal{M}(\k)$. 

\begin{proposition}\label{prop:eigpart3}
Let $\delta_{ij} \in \{0,1\}$ be Kronecker's delta. Then 
\[M(\k) = \ell^k(\ell-2)\prod_{i = 1}^t \lambda_i^{2k_i + 1}(\lambda_i-1 - \delta_{0k_i})(\lambda_i+1).\]
\end{proposition}
\begin{proof}
Recall that we view $\Z_\ell/\ell^{k+1}\Z_\ell$ as a subring of $\T_\ell/I_{\k+1}$ in the natural way.  We have 
\[M(\k) = \sum_{D \in \Z_\ell^{\diamondsuit}(\k+1)} M(\k,D)\] where \[M(\k,D) = \#\{g \in \GL_2(\T_\ell/I_{\k+1})  \colon \det(g)  = D \mbox{ and } v(\det(g - 1)) = \k\}.\] 
On the other hand, 
\[M(\k,D) = \prod_{i = 1}^t M(k_i,D_i),\]
where $D_i$ is the image of $D$ in $\Z_\ell/\ell^{k_i+1}$ and 
\[M(k_i,D_i) = \#\{g \in \GL_2(\O_i/\m_i^{k_i +1}) \colon \det(g) = D_i \mbox{ and } v_i(\det (g - 1)) = k_i\}.\] 
Then by Proposition \ref{prop:1eigpart2}, we have
\[M(\k, D) = \prod_{i = 1}^t \lambda_i^{2k_i + 1}(\lambda_i-1 - \delta_{0k_i})(\lambda_i+1),\]
where $\lambda_i = \#\O_i/\m_i$. Thus, 
\begin{align*}
    M(\k) =& \sum_{D \in \Z_\ell^{\diamondsuit}(\k+1)}M(\k,D) \\
    =&\ell^{k}(\ell-2)\prod_{i = 1}^t \lambda_i^{2k_i + 1}(\lambda_i-1 - \delta_{0k_i})(\lambda_i+1).
    \end{align*}
    as claimed.
\end{proof}
Let $\G_\ell' = \{g \in \G_\ell \colon \det(g) - 1 \in \Z_\ell^\times\}$. The following is an equivalent formulation of Theorem \ref{thm:mainexplicit}. 
\begin{theorem}\label{thm:distrib}
Fix $\k \in \Z^t_{\geq 0}$ and let $\delta_i = 1$ if $k_i = 0$ and $\delta_i = 0$ if $k_i > 0$. Then the proportion of $g \in \G_\ell'$ such that $\T_\ell/\det(g-1)\T_\ell \simeq G_\k$ is
\[\prod_{i = 1}^t \lambda_i^{-k_i}\left(1 - \dfrac{1}{\lambda_i - 1}\right)^{\delta_i} = \frac{1}{\#G_\k}\prod_{i=1}^t\left(1 - \dfrac{1}{\lambda_i - 1}\right)^{\delta_i}.\]
\end{theorem}

\begin{remark}
The distribution above is a product of $t$ independent probability distributions, since
\[
\sum_{k = 0}^\infty \lambda^{-k}\left(1 - \dfrac{1}{\lambda - 1}\right)^{\delta_{0k}} = 1 - \frac{1}{\lambda-1} + \frac{1}{1-1/\lambda} -1 = 1\]
\end{remark}

\begin{proof}
Define 
\[\G_\ell'(\k+1):=\left\{g \in \GL_2(\T_\ell/I_{\k + 1}) \colon \det(g) \in \Z_\ell^{\diamondsuit}(\k+1)\right\}.\]
Then 
\begin{align*}
    \#\G_\ell'(\k+1) =& \#\Z_\ell^{\diamondsuit}(\k+1) \cdot\#\ker\left(\GL_2(\T_\ell/I_{\k+1}) \to (\T_\ell/I_{\k+1})^\times\right)\\
    =&\ell^k(\ell-2)\prod_{i = 1}^t \#\ker\left(\GL_2(\O_i/\m_i^{k_i+1}) \to (\O_i/\m_i^{k_i+1})^\times\right)\\
  =&\ell^k(\ell-2)\prod_{i = 1}^t \lambda_i^{4k_i}\frac{(\lambda_i^2 - 1)(\lambda_i^2 - \lambda_i)}{\lambda_i^{k_i}(\lambda_i-1)}\\
  =& \ell^k(\ell-2)\prod_{i = 1}^t \lambda_i^{3k_i+1}(\lambda_i^2 - 1)
  \end{align*}
Thus, by Proposition \ref{prop:eigpart3}, the proportion of $g \in \G_\ell'$ giving rise to $G_\k$ is
\[\prod_{i = 1}^t \dfrac{\lambda_i^{2k_i +1}(\lambda_i- 1-\delta_{0k_i})}{\lambda_i^{3k_i +1}(\lambda_i - 1)} = \prod_{i = 1}^t \lambda_i^{-k_i}\left(1 - \frac{1}{\lambda_i -1}\right)^{\delta_i}.\]
\end{proof}

For the next two corollaries, we fix a prime $\ell \nmid 6n\Disc(\T)$.
\begin{corollary}\label{cor:zero}
  As $q \to \infty$ through primes $q\not\equiv 1 \pmod\ell$, we have  
\[\P(J_{p,q}[\ell^\infty] = 0) = \prod_{i = 1}^t\left(1 - \frac{1}{\lambda_i - 1}\right).\]
\end{corollary}

\begin{corollary}\label{cor:cyclicproportion}
As $q \to \infty$ through primes $q \not\equiv 1 \pmod\ell$, the probability that $J_{p,q}[\ell^\infty]$ is a cyclic abelian group is \[\left(1 + \frac{t_1}{\ell-2}\right)\prod_{i = 1}^t\left(1 - \frac{1}{\lambda_i - 1}\right),\]
where $t_1$ is the number of $\O_i$ which are isomorphic to $\Z_\ell$ $($i.e. with $d_i= 1)$.

\end{corollary}
\begin{proof}
The $\O_i$-component of $J_{p,q}[\ell^\infty]$ can only be non-trivial and cyclic if $d_i = 1$, or in other words $\lambda_i = \ell$. Indeed, $\O_i/\m  \simeq \F_{\lambda_i}$ is isomorphic to $(\Z/\ell)^{d_i}$ as an abelian group. Let us index the factors $\O_i$ so that $d_i= 1$ for $1 \leq i \leq t_1$.  Then the probability that $J_{p,q}[\ell^\infty]$ is cyclic is 
\[\prod_{i=1}^t \frac{\lambda_i -2}{\lambda_i - 1} + \sum_{i = 1}^{t_1}\frac{1}{\ell-1}\prod_{j \neq i}^t\frac{\lambda_j - 2}{\lambda_j - 1}\]
\[ = \kappa + \sum_{i = 1}^{t_1} \kappa\cdot \dfrac{\ell-1}{\ell-2}\cdot \frac{1}{\ell-1} = \kappa + \frac{t_1}{\ell-2}\kappa = \left(1 + \frac{t_1}{\ell-2}\right)\kappa,\]
where $\kappa = \prod_{i = 1}^t\frac{\lambda_i - 2}{\lambda_i - 1}$.
\end{proof}

What about analogues of the results for $q \equiv 1\pmod{\ell}$? In principle one can write a closed formula for the distribution of $J_{p,q}[\ell^\infty]$ (as in Theorem \ref{thm:distrib})  as $q \to \infty$ varies through primes such that $v_\ell(p-1) = k > 0$, for any fixed $k$, but the formulas seemed very complicated to us. For our purposes, we will be happy to simply give analogues of Corollaries \ref{cor:zero} and \ref{cor:cyclicproportion}.

\begin{proposition}
As $q \to \infty$ through $q \equiv1 \pmod{\ell}$, the probability that $J_{p,q}[\ell^\infty] = 0$ is 
\[\prod_{i = 1}^t \left(1 - \dfrac{\lambda_i}{\lambda_i^2-1}\right).\] 
\end{proposition}
\begin{proof}
First let $\O = \O_i$ and $\lambda = \lambda_i$. Then the number of $M \in \GL_2(\O/\m)$ with  $\det(M) = 1$ and $\det(M-1) = 0$ is $(\lambda-1)^2 + (\lambda-1) + \lambda = \lambda^2$, by (\ref{eq:sum}). The total number of $M$ with determinant $1$ is 
\[\dfrac{\#\GL_2(\F_\lambda)}{\#\F_\lambda^\times} = \dfrac{(\lambda^2 - 1)(\lambda^2 - \lambda)}{\lambda-1} = \lambda(\lambda-1)(\lambda+1)\]
Thus the probability that such an $M$ has $\det(M-1) \neq 0$ is $1 - \frac{\lambda}{\lambda^2-1}$. The proposition now follows by taking the product over all $i$. 
\end{proof}

\begin{proposition}\label{prop:1modl}
As $q \to \infty$ through $q \equiv1 \pmod{\ell}$, the probability that $J_{p,q}[\ell^\infty]$ is cyclic is
\[\left(1 + \frac{t_1\ell}{\ell^2 -\ell-1}\right)\prod_{i = 1}^t\left(1 - \frac{\lambda_i}{\lambda_i^2 - 1}\right)\]
\end{proposition}
\begin{proof}
The probability is
\[\prod_{i=1}^t \frac{\lambda_i^2 - \lambda_i -1}{\lambda_i^2 - 1} + \sum_{i = 1}^{t_1}\frac{\ell}{\ell^2-1}\prod_{j \neq i}^t\frac{\lambda_j^2 - \lambda_j - 1}{\lambda_j^2 - 1}\]
\[ = \kappa + \sum_{i = 1}^{t_1} \kappa\cdot \dfrac{\ell^2 -1 }{\ell^2 - \ell -1}\cdot \frac{\ell}{\ell^2-1} = \kappa + \frac{t_1\ell}{\ell^2 - \ell - 1}\kappa = \left(1 + \frac{t_1\ell}{\ell^2 -\ell-1}\right)\kappa,\]
where $\kappa = \prod_{i = 1}^t\frac{\lambda_i^2 - \lambda_i - 1}{\lambda_i^2 - 1}$.
\end{proof}

Putting together Corollary \ref{cor:cyclicproportion} and Proposition \ref{prop:1modl}, we obtain the following, which is Theorem \ref{thm:introcylicsylow}.

\begin{corollary}\label{cor:cyclicpropfixedl}
As $q \to \infty$, the probability that $J_{p,q}[\ell^\infty]$ is cyclic is 
\[\frac{\ell-2 + t_1}{\ell-1}\prod_{i = 1}^{t}\left(1 - \frac{1}{\lambda_i-1}\right) + \left(\frac{\ell^2 + (t_1 - 1)\ell - 1}{\ell^3 - 2\ell^2 + 1}\right)\prod_{i = 1}^{t}\left(1 - \frac{\lambda_i}{\lambda_i^2 -1}\right).\]
As $\ell \to \infty$, this is $1 - O(\ell^{-2})$, with the implied constant depending only on $p$.  
\end{corollary}
\begin{proof}
Call the desired probability $P$.   Then by Dirichlet's theorem on primes in arithmetic progressions, we have $P = (1-\frac{1}{\ell-1})P_A + \frac{1}{\ell-1}P_B$, where $P_A$ and $P_B$ are the probabilities computed in Corollary \ref{cor:cyclicproportion} and \ref{prop:1modl}, respectively.  This gives the claimed formula. To estimate the first term, we may ignore the factors with $\lambda_i = \ell^{d_i} > \ell$. Taking $x = \ell-1$, we compute
\[P_A \approx \left(1 + \frac{t_1}{x-1}\right)\left(1 - \frac{1}{x}\right)^{t_1} = 1 + O(x^{-2}) = 1 + O(\ell^{-2}).\]
Thus
\[P = \left(1 - \frac{1}{\ell-1}\right) + O(\ell^{-2}) + \frac{1}{\ell-1}\left(1 + O(\ell^{-1})\right) = 1 + O(\ell^{-2}),\]
as desired.
\end{proof}

Since $\prod_\ell (1 - O(\ell^{-2})) > 0$, this gives an explicit positive upper bound on the proportion of primes $p$ such that $J_{p,q}$ is cyclic.  It also suggests that a positive proportion of the groups $J_{p,q}$ are cyclic, but we cannot deduce this from our results. We at least have the following, which is Theorem \ref{thm:introcyclic}.
\begin{theorem}\label{thm:cyclicprop}
Let $S$ be the set of primes dividing $6n\Disc(\T)$.  The probability that $\bigoplus_{\ell \notin S} J_{p,q}[\ell^\infty]$ is cyclic is at most 
\[\displaystyle \prod_{\ell \notin S} \left(\frac{\ell-2 + t_1(\ell)}{\ell-1}\prod_{i = 1}^{t(\ell)}\left(1 - \frac{1}{\lambda_i-1}\right) + \left(\frac{\ell^2 + (t_1(\ell) - 1)\ell - 1}{\ell^3 - 2\ell^2 + 1}\right)\prod_{i = 1}^{t(\ell)}\left(1 - \frac{\lambda_i}{\lambda_i^2 -1}\right)\right),\]
where $t(\ell)$ is the number of factors in $\T_\ell$ and $t_1(\ell)$ is the number of degree one factors of $\T_\ell$. In particular, the above Euler product is also an upper bound on the probability that $J_{p,q}$ is cyclic.
\end{theorem}
\begin{proof}
Technically speaking, this does not follow from the previous results. However, Ribet's results apply not just to a single prime, but to any set of primes not contained in $S$ \cite{RibetPac}. This shows that for any finite set of primes $\ell_1,\ldots, \ell_m$ not in $S$, the distributions of the groups $J_{p,q}[\ell_i^\infty]$ are independent of each other. Taking a limit as $m \to \infty$ leads to the desired upper bound.  
\end{proof}
In fact, Ribet's result applies even to infinite sets of primes not contained in $S$. This suggests that with some extra work one might be able to prove a uniformity estimate and hence prove: 

\begin{conjecture}\label{conj:cyclicityconjecture}
The probability that $\bigoplus_{\ell \notin S} J_{p,q}[\ell^\infty]$ is cyclic is positive and equal to
\[\displaystyle \prod_{\ell \notin S} \left(\frac{\ell-2 + t_1(\ell)}{\ell-1}\prod_{i = 1}^{t(\ell)}\left(1 - \frac{1}{\lambda_i-1}\right) + \left(\frac{\ell^2 + (t_1(\ell) - 1)\ell - 1}{\ell^3 - 2\ell^2 + 1}\right)\prod_{i = 1}^{t(\ell)}\left(1 - \frac{\lambda_i}{\lambda_i^2 -1}\right)\right) > 0.\]
\end{conjecture}

\section{Examples}\label{sec:examples}
We make our results completely explicit for the two smallest values of $p$.  One may do similar computations for any value of $p$, after computing the Hecke algebra $\T$ in sage or magma.

\subsection{Case $p = 37$}
In this case $\#X_p = n = (37-1)/12 = 3$ and hence dim $S_2(\Gamma_0(37)) = 2$. The two newforms $f_1$ and $f_2$ of level 37 have rational coefficients. It follows that $\T$ is a subring of $\Z \times \Z$, and by a discriminant computation in SageMath \cite{sagemath}, we find that it has index two (coming from the congruence $f_1 \equiv f_2 \pmod 2$). Thus $\T_\ell = \Z_\ell \times \Z_\ell$ for $\ell > 2$.  By Theorem \ref{thm:distrib}, if $\ell > 3$ is fixed and the primes $q \not\equiv 1 \pmod \ell$ go to $\infty$, the probability that $J_{37}(q)[\ell^\infty] \simeq \Z/\ell^{k_1} \times \Z/\ell^{k_2}$ (as a $\T_\ell$-module) is
\begin{equation}\label{eq:37dist}
    \frac{1}{\ell^{k_1+k_2}}\left( \dfrac{\ell-2}{\ell-1}\right)^{\delta_{0k_1}}\left( \dfrac{\ell-2}{\ell-1}\right)^{\delta_{0k_2}}.
\end{equation}
In particular, for fixed $\ell$ and for $q \not\equiv 1\pmod{\ell}$ going to $\infty$, we have \[\P(J_{37}(q)[\ell^\infty] = 0) = \left( \dfrac{\ell-2}{\ell-1}\right)^2.\] 
To determine the distribution of the underlying abelian groups (ignoring the $\T_\ell$-module structure), treat the two factors as unordered. For example, the probability that $J_{37}(q)[\ell^\infty] \simeq \Z/\ell \times \Z/\ell^2$ as abelian groups is 
\[\P\left(J_{37}(q)[\ell^\infty] \simeq \Z/\ell \times \Z/\ell^2\right) + \P\left(J_{37}(q)[\ell^\infty] \simeq \Z/\ell^2 \times \Z/\ell\right) = 2 \cdot \frac{1}{\ell^3}.\]
\begin{example}{\em
The $(i,j)$ entry of the following matrix is the number of primes $q \not\equiv1 \pmod5$ less than $14000712$ with $J_{37}(q)[5^\infty] \simeq \Z/5^{i-1} \times \Z/5^{j-1}$ (as abelian groups), with $i \leq j$.
\[
\begin{pmatrix}
409362 &  218950 &  43483 & 8787& 1829& 359& 69& 12& 1& 0\\
0& 29077& 11591& 2239& 456& 103& 22& 7& 2& 0\\
0& 0& 1132& 465& 92& 15& 4& 1& 0& 0\\
0& 0& 0& 45& 20& 1& 1& 0& 0& 0\\
0& 0& 0& 0& 4& 0& 0& 0& 0& 0
\end{pmatrix}
\]
We see that the group $\Z/5 \times \Z/25$, say, shows up a proportion of $\frac{11591}{728129} \approx  .015919$ of the time, compared to the true asymptotic proportion $\frac{2}{5^3} = .016$. 
}
\end{example}

For any $\ell > 3$, Corollary \ref{cor:cyclicpropfixedl} says that the proportion of $q$ for which $J_{37}(q)[\ell^\infty]$ is cyclic is 
\[1 - \dfrac{(\ell+2)(\ell^2 -\ell-1)}{(\ell-1)^3(\ell+1)^2}.\]
If we restrict to $q \not \equiv 1\pmod\ell$, then by Corollary \ref{cor:cyclicpropfixedl} this proportion is 
\[\left(1 + \frac{2}{\ell-2}\right)\left(\dfrac{\ell-2}{\ell-1}\right)^2 = \dfrac{\ell(\ell-2)}{(\ell-1)^2}.\]
For example, when $\ell = 5$, this is $\frac{15}{16} = .9375$, which can be compared with the proportion $.937817\ldots$ computed from the data in the matrix above. If we allow primes up to $19000853$, then this proportion becomes $.937752\ldots$, consistent with the convergence to $.9375$.

Conjecture \ref{conj:cyclicityconjecture} says that the probability that $J_{37}(q)[1/6]$ is cyclic is
\[\prod_{\ell > 3} \left(1 - \dfrac{(\ell+2)(\ell^2 -\ell-1)}{(\ell-1)^3(\ell+1)^2}\right) = .885\ldots\]
\subsection{Case $p = 61$}
We have $n = (61-1)/12 = 5$, so $\dim S_2(\Gamma_0(61)) = 4$. There are two $\Gal(\bar \Q/\Q)$-orbits of newforms $f_1$ and $f_2$. The fields generated by their Fourier coefficients are $\Q$ and $K = \Q[x]/(x^3 - 30x - 2)$, respectively. The latter is a cubic field of discriminant $148 = 2^2\cdot 37$. Sage reports that the Hecke algebra $\T$ has discriminant $2^4 \cdot 37$ and hence is index two in the maximal order $\Z \times \O_K$. The non-maximality of $\T$ comes from the congruence $f_1 \equiv f_2 \pmod \lambda$, where $\lambda$ is the unique prime of $K$ above $2$.  

Thus, for any fixed $\ell\notin\{2,5,37\}$, we have \[\T_\ell \simeq \Z_\ell \times \bigoplus_{\lambda \mid \ell} \O_\lambda,\] 
where the sum is over the primes $\lambda$ of $\O_K$ above $\ell$ and $\O_\lambda$ is the completion of $\O_K$ at $\lambda$. Note that $\O_\lambda$ is unramified of degree $f_\lambda$, where $f_\lambda$ is the residual degree: $\O_K/\lambda \simeq \F_{\ell^{f_\lambda}}$.  The residual degrees are either $(1,1,1)$, $(1,2)$, or $(3)$, depending on how $\ell$ splits in $\O_K$. By Cebotarev's density theorem, the proportion of primes $\ell$ with the given splitting type is $1/6$, $1/2$, and $1/3$, respectively.

For example, the prime $13$ splits as $\mathfrak{l}_1\mathfrak{l}_2$ in $\O_K$, with $\mathfrak{l}_i$ having norm $13^i$. Thus, $\lambda_1 = \lambda_2 = 13$ and $\lambda_3 = 13^2$, and  \[J_{61}(q)[13^\infty] \simeq (\Z/13^a) \times(\Z/13^b) \times (\Z/13^c)^2\]
for integers $a,b,c \geq 0$. For primes $q \not\equiv1\pmod{13}$, the probability for the tuple $(a,b,c)$ can be read off of Theorem \ref{thm:mainexplicit}. For example, the probability that $J_{61}(q)[13^\infty] \simeq (\Z/13)^3$, as abelian groups, is 
\[2\cdot \dfrac{1}{13}\dfrac{11}{12}\dfrac{1}{13^2} \approx .000834,\]
corresponding to the two tuples $(1,0,1)$ and $(0,1,1)$. In Table \ref{table:N=61}, we compare the asymptotic proportions with the observed proportion for the first $62772$ primes $q \not\equiv 1\pmod{13}$. 

\begin{table}[h]
	\caption{Distribution of the group $J_{61}(q)[13^\infty] \simeq \prod_{i = 1}^4(\Z/13^{a_i})$, for $q\not\equiv1\pmod{13}$}\label{table:N=61}

\begin{tabular}{|c|c|c|}
	\hline
$(a_1,a_2,a_3,a_4)$ & \mbox{Limiting proportion} & Observed proportion ($q \leq 861997$)\\
\hline

$(0,0,0,0)$ & $20207/24192 \approx .83527$ &  $\approx.8356$\\
\hline
$(0,0,0,1)$ & $1837/13104 \approx .1401$ & $\approx.1398$\\
\hline
$(0,0,1,1)$ & $1849/170352\approx .01085$ & $\approx.01096$\\
\hline
$(0,0,0,2)$ & $1837/170352 \approx .01078$ & $\approx.01067$\\
\hline
$(0,0,1,2)$ & $167/184548 \approx .0009049$& $\approx.0009239$ \\
\hline
$(0,1,1,1)$ & $11/13182 \approx .00083447$ & $\approx.000860$\\
\hline
$(0,0,0,3)$ & $1837/2214576 \approx .0008295$ & $\approx.0008283$\\
\hline
\end{tabular}
\end{table}

The shape of the Euler factor in Corollary \ref{cor:cyclicpropfixedl} depends on the splitting type of $\ell$ in $\Z \times \O_K$. For the types $(1,1,1,1)$, $(1,1,2)$, and $(1,3)$ we compute the following Euler factors:
\begin{align*}
f_1(\ell) &= 1 - \dfrac{6\ell^7 - 2\ell^6 - 43\ell^5 + 17\ell^4 + 92\ell^3 - 2\ell^2 - 79\ell - 34}{(\ell - 1)^5(\ell + 1)^4}\\
f_2(\ell) &= 1 - \dfrac{(\ell+1)^3(2\ell^7 - 5\ell^5 - 3\ell^4 + 4\ell^2 + 7\ell + 2)}{(\ell - 1)^4(\ell^2 + 1)}\\
f_3(\ell) &= 1 - \frac{\ell^4 - \ell^3 + \ell - 2}{(\ell - 1)^2 (\ell + 1) (\ell^2 - \ell + 1)(\ell^2 + \ell + 1)} 
\end{align*}
Since $2\cdot3\cdot5\cdot37 = 1110$, Conjecture \ref{conj:cyclicityconjecture} predicts that  $J_{61}(q)[1/1110]$ is cyclic with probability 
\[\prod_{ \ell\O_K = \lambda_1\lambda_2\lambda_3} f_1(\ell)\prod_{\ell\O_K = \lambda_1\lambda_2} f_2(\ell)\prod_{ \ell\O_K = (\ell)}f_3(\ell) = .9544\ldots.\]
For primes $q$ up to 861997, we found the proportion to be $\frac{65325}{68492} \approx .9537$.

\bibliography{references}

\begin{thebibliography}{{The}20}

\bibitem[CDF97]{ConreyDukeFarmer}
J.~B. Conrey, W.~Duke, and D.~W. Farmer.
\newblock The distribution of the eigenvalues of {H}ecke operators.
\newblock {\em Acta Arith.}, 78(4):405--409, 1997.

\bibitem[CL84]{CL}
H.~Cohen and H.~W. Lenstra, Jr.
\newblock Heuristics on class groups of number fields.
\newblock In {\em Number theory, {N}oordwijkerhout 1983 ({N}oordwijkerhout,
  1983)}, volume 1068 of {\em Lecture Notes in Math.}, pages 33--62. Springer,
  Berlin, 1984.

\bibitem[CLP15]{CLP}
Julien Clancy, Timothy Leake, and Sam Payne.
\newblock A note on {J}acobians, {T}utte polynomials, and two-variable zeta
  functions of graphs.
\newblock {\em Exp. Math.}, 24(1):1--7, 2015.

\bibitem[Eme02]{Emerton02}
Matthew Emerton.
\newblock Supersingular elliptic curves, theta series and weight two modular
  forms.
\newblock {\em J. Amer. Math. Soc.}, 15(3):671--714, 2002.

\bibitem[Gar15]{Garton}
Derek Garton.
\newblock Random matrices, the {C}ohen-{L}enstra heuristics, and roots of
  unity.
\newblock {\em Algebra Number Theory}, 9(1):149--171, 2015.

\bibitem[Li19]{Li}
Wen-Ching~Winnie Li.
\newblock {\em Zeta and {$L$}-functions in number theory and combinatorics},
  volume 129 of {\em CBMS Regional Conference Series in Mathematics}.
\newblock American Mathematical Society, Providence, RI, 2019.
\newblock Published for the Conference Board of the Mathematical Sciences.

\bibitem[M\'20]{meszaros}
Andr\'{a}s M\'{e}sz\'{a}ros.
\newblock The distribution of sandpile groups of random regular graphs.
\newblock {\em Trans. Amer. Math. Soc.}, 373(9):6529--6594, 2020.

\bibitem[Maz77]{MazurEisenstein}
B.~Mazur.
\newblock Modular curves and the {E}isenstein ideal.
\newblock {\em Inst. Hautes \'{E}tudes Sci. Publ. Math.}, (47):33--186 (1978),
  1977.
\newblock With an appendix by Mazur and M. Rapoport.

\bibitem[Mum08]{Mumford}
David Mumford.
\newblock {\em Abelian varieties}, volume~5 of {\em Tata Institute of
  Fundamental Research Studies in Mathematics}.
\newblock Published for the Tata Institute of Fundamental Research, Bombay; by
  Hindustan Book Agency, New Delhi, 2008.
\newblock With appendices by C. P. Ramanujam and Yuri Manin, Corrected reprint
  of the second (1974) edition.

\bibitem[Rib97]{RibetPac}
Kenneth~A. Ribet.
\newblock Images of semistable {G}alois representations.
\newblock Number Special Issue, pages 277--297. 1997.
\newblock Olga Taussky-Todd: in memoriam.

\bibitem[Ser97]{SerreHecke}
Jean-Pierre Serre.
\newblock R\'{e}partition asymptotique des valeurs propres de l'op\'{e}rateur
  de {H}ecke {$T_p$}.
\newblock {\em J. Amer. Math. Soc.}, 10(1):75--102, 1997.

\bibitem[Shi58]{ShimuraEichler}
Goro Shimura.
\newblock Correspondances modulaires et les fonctions {$\zeta $} de courbes
  alg\'{e}briques.
\newblock {\em J. Math. Soc. Japan}, 10:1--28, 1958.

\bibitem[Sil86]{AEC}
Joseph~H. Silverman.
\newblock {\em The arithmetic of elliptic curves}, volume 106 of {\em Graduate
  Texts in Mathematics}.
\newblock Springer-Verlag, New York, 1986.

\bibitem[{The}20]{sagemath}
{The Sage Developers}.
\newblock {\em {S}ageMath, the {S}age {M}athematics {S}oftware {S}ystem
  ({V}ersion 9.2)}, 2020.
\newblock {\tt https://www.sagemath.org}.

\bibitem[Woo17]{WoodSandpile}
Melanie~Matchett Wood.
\newblock The distribution of sandpile groups of random graphs.
\newblock {\em J. Amer. Math. Soc.}, 30(4):915--958, 2017.

\end{thebibliography}
\bibliographystyle{alpha}

\end{document}